\documentclass[11pt]{article}
\usepackage{amsmath}
\usepackage{cases}
\usepackage{amssymb}
\usepackage{amsfonts}

\usepackage{bbm}
\usepackage{amsmath,amsthm,amssymb,amscd}
\usepackage{latexsym}
\usepackage{hyperref}
\usepackage[numbers,sort&compress]{natbib}
\usepackage{hypernat}
\allowdisplaybreaks

\newtheorem{theorem}{Theorem}[section]
\newtheorem{corollary}[theorem]{Corollary}

\newtheorem{remark}[theorem]{Remark}

\theoremstyle{definition} \theoremstyle{remark}
\numberwithin{equation}{section}

\begin{document}

\title{{\bf A note on closedness of the sum of two closed subspaces in a Banach space\thanks{The work was supported by
NSFC (11461034), and the Program for Cultivating Young
Scientist of Jiangxi Province (20133BCB23009).}}}
\date{}
\author{Zhe-Ming Zheng, Hui-Sheng Ding\thanks{Corresponding author. E-mail address:
dinghs@mail.ustc.edu.cn (H.-S. Ding)} \\
{\small { \it College of Mathematics and Information Science,
Jiangxi Normal University}}\\{\small {\it Nanchang, Jiangxi 330022,
People's Republic of China }}}
\maketitle

\begin{abstract}
Let $X$ be a Banach space, and $M,N$ be two closed subspaces of $X$. We present several necessary and sufficient conditions for the closedness of $M+N$ ($M+N$ is not necessarily direct sum).


\end{abstract}

\section{Introduction}
Let $X$ be a Banach space, and $M,N$ be two closed subspaces of $X$. Then, $M+N$ is not necessarily closed in $X$ even if $X$ is a Hilbert space and $M\cap N=\{0\}$ (see, e.g., \cite[p.145, Exercise 9]{rudin}). So, to study when $M+N$ is closed in $X$ is always an interesting problem.

For the case of $M\cap N=\{0\}$, a necessary and sufficient condition for $M+N$ being closed in $X$ is given by Kober \cite{Kober} as follows:
\begin{theorem}\label{kober}Let $X$ be a Banach space, $M,N$ be two closed subspaces of $X$ and $M\cap N=\{0\}$. Then $M+N$ is closed in $X$ if and only if
there exists a constant $A>0$ such that for all $x\in M$ and $y\in N$ we have $\|x\|\leq A\cdot \|x+y\|$.
\end{theorem}

It seems that there are seldom results concerning necessary and sufficient conditions for $M+N$ being closed in $X$ in the case of $M+N$ being not necessarily direct sum. To the best of our knowledge, the first result of a necessary and sufficient condition for $M+N$ (not necessarily direct sum) being closed in $X$ is given by Luxemburg:
\begin{theorem}{\rm\cite[Theorem 2.5]{Luxemburg}}\label{Luxemburg} Let $X$ be a Banach space, and $M,N$ be two closed subspaces of $X$. Then $M+N$ is closed in $X$ if and only if $T:M\times N\to X; (m,n)\longmapsto m+n$ is an open mapping.
\end{theorem}
Luxemburg \cite{Luxemburg} obtain the above theorem in a more general setting. Theorem \ref{Luxemburg} is only one of the interesting results concerning this topic given by Luxemburg. We refer the reader to \cite{Luxemburg} for more details.


In addition, for the case of $X$ being a Banach lattice or a Hilbert space, there has been of great interest for some researchers to study if the sum of two closed subspaces of $X$ is still closed. We refer the reader to \cite{Luxemburg,l1,l2,voigt} and references therein for the case of $X$ being a Banach lattice or a Fr\'{e}chet space and to \cite{H1,H2} and references therein for the case of $X$ being a Hilbert space.

This short note is also devoted to this problem for the case of $X$ being a general Banach space. As one will see, we give a Kober-like theorem for the case of $M+N$ being not necessarily direct sum, and show that a necessary condition in the classical textbook \cite{rudin} is also sufficient (see Remark \ref{remark1}).

\section{Main results}

\begin{theorem}\label{main}Let $X$ be a Banach space, and $M,N$ be two closed subspaces of $X$. Then the following assertions are equivalent:
\begin{itemize}
\item[(i)] $M+N$ is closed in $X$;
\item[(ii)] $(M+N)/N$ is closed in $X/N$;
\item[(iii)] there exists a constant $K>0$ such that for every $x\in M+N$, there is a decomposition $x=m+n$ such that
$$\|m\|\leq K\cdot \|x\|,$$
where $m\in M$ and $n\in N$;
\item[(iv)] $T:M\times N\to M+N; (m,n)\longmapsto m+n$ is an open mapping.
\end{itemize}
\end{theorem}
\begin{proof}"(i) $\Longrightarrow$ (ii)". It is obvious.

"(ii) $\Longrightarrow$ (iii)". Define a mapping $\phi:(M+N)/N\to M/(M\cap N)$ by
$$\phi(x+N)=m+(M\cap N),$$
where $x=m+n\in M+N $, $m\in M$ and $n\in N$. It is easy to see that $\phi$ is well-defined. Moreover, $\phi$ is linear and bijective. Noting that
$$\|\phi(x+N)\|=\|m+(M\cap N)\|\geq \|m+N\|= \|x+N\|,$$
we conclude that $\phi^{-1}$ is a bounded linear operator from $M/(M\cap N)$ to $(M+N)/N$.
Since $(M+N)/N$ and $M/(M\cap N)$ are both Banach spaces, it follows from the open mapping theorem that $\phi$ is also a bounded linear operator from $(M+N)/N$ to $M/(M\cap N)$. Taking $K=\|\phi\|+1$, the assertion (iii) follows. In fact, letting $x=m'+n'\in M+N$ and $x\neq 0$, where $m'\in M$ and $n'\in N$, we have
$$\|m'+(M\cap N)\|=\|\phi(x+N)\|\leq \|\phi\|\cdot \|x+N\|\leq \|\phi\|\cdot \|x\|<K\|x\|.$$
Then, there exists $y\in M\cap N$ such that
$$\|m'+y\|<K\|x\|.$$
Letting $m=m'+y$ and $n=n'-y$, we get $x=m+n$ and $\|m\|<K\|x\|$.

"(iii) $\Longrightarrow$ (iv)". It is easy to see that $$kerT=\{(x,-x):x\in M\cap N\}.$$
Let $\pi$ be the quotient map from $M\times N$ to $(M\times N)/ kerT$, and $\widetilde{T}:(M\times N)/ kerT \to M+ N$ be defined as follows
$$\widetilde{T}[(m,n)+kerT]=m+n,\quad (m,n)\in M\times N.$$
Then $\widetilde{T}$ is linear and bijective. For every $(m,n)\in M\times N$, by (iii), there exist $m'\in M$ and $n'\in N$ such that $m+n=m'+n'$ and
$$\|m'\|\leq K\|m+n\|,$$
which yields that
$$\|m'\|+\|n'\|\leq (2K+1)\|m+n\|.$$
Then, we have
$$\|\widetilde{T}[(m,n)+kerT]\|=\|m+n\|\geq \frac{\|m'\|+\|n'\|}{2K+1}\geq \frac{1}{2K+1}\|(m,n)+kerT\|,$$
which means that $\widetilde{T}$ is an open mapping. Combing this with the fact that $\pi$ is open, we conclude that $T=\widetilde{T}\circ \pi$ is also open.

"(iv) $\Longrightarrow$ (i)". As noted in the Introduction, (i) is equivalent to (iv) has been shown by Luxemburg using a more general setting. Here, we give a different proof (maybe a more direct proof in the setting of Banach spaces).

Let $\pi,\ kerT,\ \widetilde{T}$ be as in the proof of "(iii) $\Longrightarrow$ (iv)". For every $(m,n)\in M\times N$ and $x\in M\cap N$, there holds
$$\|m+n\|\leq \|m+x\|+\|n-x\|=\|(m+x,n-x)\|=\|(m,n)+(x,-x)\|,$$
which yields $$\|\widetilde{T}[(m,n)+kerT]\|=\|m+n\|\leq \inf_{x\in M\cap N}\|(m,n)+(x,-x)\|=\|(m,n)+kerT\|,$$
i.e., $\|\widetilde{T}\|\leq 1.$ On the other hand, since $\pi:M\times N \to (M\times N)/ kerT$ is continuous and $T$ is an open mapping, for every open set $U\subset (M\times N)/ kerT$,
$$\widetilde{T}(U)=T(\pi^{-1}(U))$$
is also an open set. Thus, $\widetilde{T}$ is an open mapping, which means that $\left(\widetilde{T}\right)^{-1}$ is continuous, and so bounded. Now, we conclude that as normed linear spaces, $M+N$ and $(M\times N)/ kerT$ are topological isomorphic. Then, it follows that $(M\times N)/ kerT$ is a Banach space that $M+N$ is also a Banach space. This completes the proof.
\end{proof}

\begin{remark}\label{remark1}\rm
 In the classical textbook \cite{rudin} (see p.137, Theorem 5.20), it has been shown that (iii) is a necessary condition for (i) by using the open mapping theorem. Here, we show that (iii) is also a sufficient condition for (i). In fact, (i) is equivalent to (iii) is a Kober-like result for the case of $M+N$ being not necessarily direct sum. Moreover, we will give a direct proof of "(iii) $\Longrightarrow$ (i)" in the following. We think that it may be of interest for some readers. Here is our proof:

Let $\{x_j\}_{j=1}^{\infty}\subset M+N$ and $x_j\to x$ in $X$ as $j\to\infty$. Then, we can choose a subsequence $\{x_k\}$ of $\{x_j\}$ such that
$$\|x_{k+1}-x_k\| \leq \frac{1}{2^k\cdot K},\quad k=1,2\ldots.$$
By taking $x=x_2-x_1$ in the assertion (iii), there exist $m_1\in M$ and $n_1\in N$ such that $x_2-x_1=m_1+n_1$ and
$$\|m_1\|\leq K\cdot \|x_2-x_1\|\leq \frac{1}{2}.$$
Similarly, by taking $x=x_3-x_2$ in the assertion (iii), there exist $m_2\in M$ and $n_2\in N$ such that $x_3-x_2=m_2+n_2$ and
$$\|m_2\|\leq K\cdot \|x_3-x_2\|\leq \frac{1}{2^2}.$$
Continuing by this way, we get two sequences $\{m_k\}\subset  M$ and $\{n_k\}\subset  N$ such that
$$x_{k+1}-x_k=m_k+n_k,\quad k=1,2\ldots,$$
and $$\|m_k\|\leq  \frac{1}{2^k},\quad k=1,2\ldots.$$
Then, we have $\sum\limits_{k=1}^{\infty}\|m_k\|<\infty.$ Also, we can get $\sum\limits_{k=1}^{\infty}\|n_k\|<\infty.$ Since $M$ and $N$ are both Banach spaces, there exist $m\in M$ and $n\in N$ such that
$$m=\sum_{k=1}^{\infty}m_k,\quad n=\sum_{k=1}^{\infty}n_k.$$
Recalling that $x_k\to x$, we get
$$x-x_1=\sum_{k=1}^{\infty}(x_{k+1}-x_k)=m+n,$$
which yields that $x=x_1+m+n\in M+N$.
\end{remark}

\begin{corollary}\label{c2}Let $X$ be a Banach space, and $M,N$ be two closed subspaces of $X$. Then the following assertions are equivalent:
\begin{itemize}
\item[(a)] $M+N$ is closed in $X$;
\item[(b)] $(M+N)/(M\cap N)$ is closed in $X/(M\cap N)$.
\end{itemize}
\end{corollary}
\begin{proof}Noting that $(M+N)/(M\cap N)=M/(M\cap N)+N/(M\cap N)$, it follows from Theorem \ref{main} that the closeness of $(M+N)/(M\cap N)$ is equivalent to the closedness of
$$[(M+N)/(M\cap N)]/[M/(M\cap N)].$$
On the other hand, it is not difficult to show that $(M+N)/M$ is isometric to $[(M+N)/(M\cap N)]/[M/(M\cap N)]$, and so their closedness are equivalent. Thus, the closedness of $(M+N)/(M\cap N)$ is equivalent to the closedness of $(M+N)/M$. Again by Theorem \ref{main}, we complete the proof.
\end{proof}

\begin{remark}\rm
By Corollary \ref{c2}, whenever we find an example of non-direct sum $M+N$, which is not closed, we can get an example of direct sum $M/(M\cap N)+N/(M\cap N)$, which is still not closed.
\end{remark}



\end{document}